\documentclass[12pt]{amsart}
\usepackage{amsfonts}
\usepackage{ifthen}
\usepackage{amsthm}
\usepackage{amsmath}
\usepackage{graphicx}
\usepackage{amscd,amssymb,amsthm}

\newcounter{minutes}
\divide\time by 60
\newcounter{hours}
\multiply\time by 60 \addtocounter{minutes}{-\time}

\title{On Kaluza's sign criterion  for reciprocal power series}

\author{\'Arp\'ad Baricz}
\author{Jetro Vesti}
\author{Matti Vuorinen}

\address{Department of Economics, Babe\c{s}-Bolyai University,
Cluj-Napoca 400591, Romania} \email{bariczocsi@yahoo.com}
\address{Department of Mathematics, University of Turku, Turku 20014,
Finland} \email{jejove@utu.fi}
\address{Department of Mathematics, University of Turku, Turku 20014,
Finland} \email{vuorinen@utu.fi}


\newtheorem{theorem}[equation]{Theorem}

\newtheorem{proposition}[equation]{Proposition}
\newtheorem{example}[equation]{Example}

\numberwithin{equation}{section}

\begin{document}

\def\thefootnote{}
\footnotetext{ \texttt{File:~\jobname .tex,
          printed: \number\year-0\number\month-\number\day,
          \thehours.\ifnum\theminutes<10{0}\fi\theminutes}
} \makeatletter\def\thefootnote{\@arabic\c@footnote}\makeatother

\maketitle

\begin{abstract}
T. Kaluza has given a criterion for the signs of the power series of
a function that is the reciprocal of another power series. In this
note the sharpness of this condition is explored and various
examples in terms of the Gaussian hypergeometric series are given. A
criterion for the monotonicity of the quotient of two power series
due to M. Biernacki and J. Krzy\.z is applied.
\end{abstract}

\section{Introduction}

In this paper we are mainly interested on the class of Maclaurin
series $\sum_{n\geq0}a_nx^n,$ which are convergent for
$x\in\mathbb{R}$ such that $|x|<r.$ Throughout in the paper
$\{a_n\}_{n\geq0}$ is a sequence of real numbers and $r>0$ is the
radius of convergence. Note that if $f(x)=\sum_{n\geq0}a_nx^n$ and
$g(x)=\sum_{n\geq0}b_nx^n$ are two Maclaurin series with radius of
convergence $r,$ then their product
$h(x)=f(x)g(x)=\sum_{n\geq0}c_nx^n$ has also radius of convergence
$r$ and Cauchy's product rule gives the coefficients $c_n$ of $h(x)$
as
\begin{equation}\label{Cauchy}
c_n= \sum_{k=0}^{n} a_kb_{n-k},
\end{equation}
known as the convolution of $a_n$ and $b_n.$ If $g(x)$ never
vanishes, also the quotient $q(x)= f(x)/g(x)=\sum_{n\geq0}q_nx^n$ is
convergent with radius of convergence $r$ and we obtain the rule for
the coefficients $q_n$ by interchanging $a$ and $c$ in
(\ref{Cauchy})
$$q_n =  (a_n - \sum_{k=0}^{n-1} q_k b_{n-k})/b_0.$$
We note that a special case of the above relation when $a_0=1$ and
$0=a_1=a_2=\dots$ yields the following result.

\begin{proposition}\label{ratkaiseminen}
Suppose that $g(x)=\sum_{n\geq0}b_n x^n$ with  $b_0\neq0$ and
${1}/{g(x)}=\sum_{n\geq0} q_n x^n.$ In order to solve $q_n$ we need
to know $b_0,b_1,b_2,\dots ,b_n$.
\end{proposition}

\begin{proof}[\bf Proof]
Since
$$\frac{1}{b_0+b_1 x+b_2 x^2+{\dots}+b_nx^n+{\dots}}
=q_0+q_1 x+q_2 x^2+{\dots}+q_n x^n+{\dots},$$ we just need to solve
the linear equations
$$\left\{\begin{array}{ll}
1=b_0 q_0\\ 0=b_1 q_0+b_0 q_1 \\ 0= b_2 q_0+b_1 q_1+b_0 q_2\\ \vdots
\\ 0= \sum_{k=0}^n b_k q_{n-k}\\
\end{array} \right.
\Longleftrightarrow  \left\{\begin{array}{ll}
q_0={1}/{b_0}\\ q_1=(-b_1 q_0)/b_0 \\
q_2=(-b_2 q_0-b_1 q_1)/b_0\\ \vdots \\
q_n=(-\sum_{k=1}^n b_k q_{n-k})/b_0 \end{array} \right..$$ Thus,
$q_n=\phi(b_0,b_1,\dots,b_n),$ where $\phi$ is some function.
\end{proof}

In 1928 Theodor Kaluza\footnote{In passing we remark that he was a
German mathematician interested in Physics, where his name is
associated with so called Kaluza-Klein theory.} \cite{k} proved the
following theorem.

\begin{theorem}\label{Kaluza1}
Let $f(x)=\sum_{n\geq0}a_nx^n$ be a convergent Maclaurin series with
radius of convergence $r>0.$ If $a_n>0$ for all $n\in\{0,1,\dots\}$
and the sequence $\{a_n\}_{n\geq0}$ is log-convex, that is, for all
$n\in\{1,2,\dots\}$
\begin{equation}\label{Kaluza2}
a_n^2 \leq a_{n-1}a_{n+1},
\end{equation}
then the coefficients $b_n$ of the reciprocal power series $1/f(x) =
\sum_{n\geq0}b_nx^n$  have the following properties: $b_0=1/a_0>0$
and $b_n\leq 0$ for all $n\in\{1,2,\dots\}.$
\end{theorem}

In what follows we say that a power series has the {\em Kaluza sign
property} if the coefficients of its reciprocal power series are all
non-positive except the constant term. Theorem \ref{Kaluza1} then
says that if the power series $f(x)$ has positive and log-convex
coefficients, then $f(x)$ has the Kaluza sign property. For a short
proof of Theorem \ref{Kaluza1} see \cite{c}. This result is also
cited in \cite[p. 68]{ha} and \cite[p. 13]{h}. Note that Theorem
\ref{Kaluza1} in Jurkat's paper \cite{j} is attributed to Kaluza and
Szeg\H o, however Szeg\H o \cite{sz} attributes this result to
Kaluza. We also note that this result implies, in particular, that
the function $x\mapsto 1/f(x)$ is decreasing on $(0,r).$ This
observation is also clear because $x\mapsto f(x)$ is increasing on
$(0,r).$ It is also important to note here that Kaluza's result is
useful in the study of renewal sequences, which are frequently
applied in probability theory. For more details we refer to the
papers \cite{hansen,horn,kendall,l} and to the references contained
therein.

We will next look at the condition (\ref{Kaluza2}) from the point of
view of power means. For fixed $a,b,t>0, $ we define the power mean
by
$$m(a,b,t)=\left(\frac{a^t + b^t}{2}\right)^{1/t}.$$
It is well-known (see for example \cite{bb}) that
$\lim\limits_{t\to0}m(a,b,t)=\sqrt{ab}$ and the function $t\mapsto
m(a,b,t)$ is increasing on $(0,\infty)$ for all fixed $a,b>0.$
Therefore for all $u>t>0$ we have
$$\sqrt{ab} \le m(a,b,t) \le m(a,b,u).$$

By observing that (\ref{Kaluza2}) is the same as $a_n
\leq\lim\limits_{t\to0} m(a_{n-1},a_{n+1},t)$ we can prove that
(\ref{Kaluza2}) is sharp in the following sense.

\begin{theorem}\label{theorem2}
Suppose that in the above theorem all the hypotheses except
(\ref{Kaluza2}) are satisfied and (\ref{Kaluza2}) is replaced with
\begin{equation}\label{Kaluza4}
a_n \leq m(a_{n-1},a_{n+1},t)
\end{equation}
where $n\in\{1,2,\dots\}$ and $t\geq 1/100.$ Then the conclusion of
Theorem \ref{Kaluza1} is no longer true.
\end{theorem}

\begin{proof}[\bf Proof]
The monotonicity with respect to $t$ yields for all
$n\in\{1,2,\dots\}$ and $u\geq t>0$
$$\left( \frac{a_{n-1}^t + a_{n+1}^t}{2}\right)^{1/t} \leq
\left(\frac{a_{n-1}^u + a_{n+1}^u}{2}\right)^{1/u}.$$ The series
$q(x)=1.999+\sum_{n\geq1}{x^n}/n$ satisfies all the hypotheses that
were made:
$$1 < \left(\frac{1.999^{1/100} + 0.5^{1/100}}{2}\right)^{100}(\approx 1.00215)
\le \left(\frac{1.999^{t} + 0.5^{t}}{2}\right)^{1/t}$$
for all $t\geq 1/100$ and generally when $n\in\{2,3,\dots\}$
$$ \frac{1}{n} < \sqrt{\frac{1}{(n-1)(n+1)}} \le
\left(\frac{\left(\frac{1}{n-1}\right)^{t} + \left(\frac{1}{n+1}\right)^{t}}{2}\right)^{1/t}$$
for all $t\geq 1/100.$ Because the series
$$\frac{1}{q(x)} = 0.50025 - 0.25025x + 0.000062594 x^2 - \ldots $$
has  a positive coefficient different from a constant term we get
our claim.
\end{proof}

Theorem \ref{theorem2} shows that it is not possible to replace the
hypothesis (\ref{Kaluza2}) with  (\ref{Kaluza4}), at least if $t\ge
{1}/{100}.$ Moreover, we note that it is easy to reduce the number
${1}/{100}$. To that end, it is enough to replace the constant
$1.999$ of the Maclaurin series $q(x)$ in the proof of Theorem
\ref{theorem2} with another constant in $(1.999, 2).$

\section{Remarks on the Kaluza sign property}

In this section we will make some general observations about power
series and Kaluza's Theorem \ref{Kaluza1}. The Gaussian
hypergeometric series is often useful for illustration purposes and
it is available at the Mathematica(R) software package which is used
for the examples. For $a,b,c$ real numbers and $|x|<1,$ it is
defined by
$$ {}_2 F_1(a,b;c;x)= \sum_{n\geq0}\frac{(a,n)(b,n)}{(c,n)n!}x^n,$$
where $(a,n)= a(a+1)...(a+n-1)=\Gamma(a+n)/\Gamma(a)$ for
$n\in\{1,2,\dots\}$ and ${(a,0)}=1,$ is the rising factorial and it
is required that $c\neq 0, -1,\dots$ in order to avoid division by
zero. Some basic properties of this series may be found in standard
handbooks, see for example \cite{olbc}.

We begin with an example which is related to Proposition
\ref{ratkaiseminen}.

\begin{example} {\rm
Let $$f(x)=\cosh{x}=\sum_{n\geq0}\frac{1}{(2n)!}x^{2n}$$ and
$$g(x)=\cos{x}=\sum_{n\geq0}\frac{(-1)^n}{(2n)!}x^{2n}.$$ Then
$$\frac{1}{f(x)} =1-\frac{x^2}{2}+\frac{5
   x^4}{24}-\frac{61
   x^6}{720}+\frac{277
   x^8}{8064}-\frac{50521
   x^{10}}{3628800}+\mathcal{O}\left(x^{
   11}\right)$$
and
$$
\frac{1}{g(x)}=1+\frac{x^2}{2}+\frac{5
   x^4}{24}+\frac{61
   x^6}{720}+\frac{277
   x^8}{8064}+\frac{50521
   x^{10}}{3628800}+\mathcal{O}\left(x^{
   11}\right).$$
Observe the similarities in the coefficients. Similarly, if
$$f(x)=\frac{\sinh(x)}{x}=\sum_{n\geq0}\frac{1}{(2n+1)!}x^{2n}$$ and
$$g(x)=\frac{\sin(x)}{x}=\sum_{n\geq0}\frac{(-1)^n}{(2n+1)!}x^{2n},$$ then
$$
\frac{1}{f(x)}= 1-\frac{x^2}{6}+\frac{7
   x^4}{360}-\frac{31
   x^6}{15120}+\frac{127
   x^8}{604800}-\frac{73
   x^{10}}{3421440}+\mathcal{O}\left(x^{
   11}\right)
$$
and
$$
\frac{1}{g(x)}=1+\frac{x^2}{6}+\frac{7
   x^4}{360}+\frac{31
   x^6}{15120}+\frac{127
   x^8}{604800}+\frac{73
   x^{10}}{3421440}+\mathcal{O}\left(x^{
   11}\right).
$$}
\end{example}

These observations are special cases of the following result.

\begin{proposition}\label{parillisetpotenssit}
Let $$f(x)=\sum_{n\geq0} a_{2n} x^{2n}\ \ \mbox{and}\ \
g(x)=\sum_{n\geq0}(-1)^n a_{2n} x^{2n},$$ where $a_{2n}>0$ for all
$n\in\{0,1,\dots\}.$ Then the coefficients of the reciprocal power
series $$\frac{1}{f(x)}=\sum_{n\geq0} b_{n}x^{n}\ \ \mbox{and}\ \
\frac{1}{g(x)}=\sum_{n\geq0}c_{n}x^{n}$$ satisfy
$b_{2n+1}=c_{2n+1}=0$ and $b_{2n}=(-1)^n c_{2n}$ for all
$n\in\{0,1,\dots\}.$
\end{proposition}

\begin{proof}[\bf Proof]
From the equation
\begin{align*}1&=(a_0+a_2 x^2+a_4x^4+\dots)(b_0+b_1 x+b_2 x^2+\dots)\\&=a_0 b_0+a_0 b_1
x+(b_0 a_2+b_2 a_0)x^2+{\dots}\\&+\left(\sum_{k=0}^{n}
b_{2k}a_{2(n-k)}\right) x^{2n}+ \left(\sum_{k=0}^{n}
b_{2k+1}a_{2(n-k)}\right)x^{2n+1}+\dots\end{align*} we get
inductively for all $n\in\{0,1,\dots\}$
$$b_1=b_3={\dots}=b_{2n+1}=0$$ and
$$b_0=\frac{1}{a_0}, b_2=\frac{1}{a_0}(-b_0
a_2),\dots, b_{2n}=\frac{1}{a_0}\left(-\sum_{k=0}^{n-1}
b_{2k}a_{2(n-k)}\right).$$ Similarly, for all $n\in\{0,1,\dots\}$ we
get
$$c_1=c_3={\dots}=c_{2n+1}=0$$ and
$$c_0=\frac{1}{a_0},c_2=\frac{1}{a_0}(c_0 a_2),\dots,
c_{2n}=\frac{1}{a_0}\left(-\sum_{k=0}^{n-1}
c_{2k}(-1)^{n-k}a_{2(n-k)}\right).$$ From these we get our claim:
$b_{2n+1}=0=c_{2n+1}$ is clear and $b_{2n}=(-1)^n c_{2n}$ follows by
induction.
\end{proof}

In the next proposition we show that log-convex sequences can be
classified into two types.

\begin{proposition}\label{kasvulasku}
If the positive sequence $\{a_n\}_{n\geq0}$ is log-convex, then the
following assertions are true:
\begin{enumerate}
\item[(1)] If $a_0\le a_1,$ then $a_0\le a_1\le a_2\le \dots;$
\item[(2)] If $a_1\leq a_0,$ then $a_0\ge a_1\ge a_2\ge\dots$
or there exists $k>0$ such that $a_0\ge a_1\ge a_2\ge \dots\ge
a_{k-1} \ge a_k$ and $a_k\le a_{k+1}\le{\dots}.$
\end{enumerate}
\end{proposition}

\begin{proof}[\bf Proof]
(1) First suppose that $a_0\le a_1.$ Then we have $a_1^2\leq a_0
a_2\le a_1 a_2,$ which implies that $a_1\leq a_2.$ Suppose that
$a_{k-1}\leq a_k$ holds for all $k\in\{1,2,\dots,n\}.$ Again from
hypothesis we get $a_{k}^2\leq a_{k-1}a_{k+1}\leq a_{k}a_{k+1},$
which implies that $a_{k}\leq a_{k+1}.$ Thus, the first claim
follows by induction.

(2) Secondly, suppose that $a_1\leq a_0.$ If there exists an index
$k>0$ such that $a_k\leq a_{k+1}$ and does not exist $s<k$ such that
$a_s\leq a_{s+1},$ then we get from hypothesis that $a_{k+1}^2\leq
a_k a_{k+2}\leq a_{k+1}a_{k+2},$ which implies that $a_{k+1}\leq
a_{k+2}.$ By induction for all $n\ge k$ we have that $a_{n}\leq
a_{n+1}.$ We also have $a_n^2\leq a_{n-1}a_{n+1}\le a_{n-1}a_{n}$
for all $n<k,$ which implies that $a_n\leq a_{n-1}$ for all $n<k.$
From these we get the last case.

If there does not exists an index $k>0$ such that $a_k\leq a_{k+1},$
then we get the former case by the same way: for all
$n\in\{1,2,\dots\}$ we have $a_{n}^2\leq a_{n-1} a_{n+1}\leq
a_{n-1}a_{n},$ which implies that $a_{n}\leq a_{n-1}$ for all
$n\in\{1,2,\dots\}.$
\end{proof}

It should be mentioned here that the previous result is related to
the following well-known result: log-concave sequences are unimodal.
Note that a sequence $\{a_n\}_{n\geq0}$ is said to be log-concave if
for all $n\geq1$ we have $a_n^2\geq a_{n-1}a_{n+1}$ and by
definition a sequence $\{a_n\}_{n\geq0}$ is said to be unimodal if
its members rise to a maximum and then decrease, that is, there
exists an index $k>0$ such that $a_0\leq a_1\leq a_2\leq{\dots}\leq
a_k$ and $a_k\geq a_{k+1}\ \geq{\dots}\geq a_n\geq{\dots}.$

We now illustrate our previous result by giving some examples.

\begin{example}
{\rm The power series
$$f_1(x)=\sum_{n\geq0}\frac{2^n+1}{2}x^n=
1+\frac{3}{2}x+\frac{5}{2}x^2+\frac{9}{2}x^3+\dots$$ is of type (1)
considered in Proposition \ref{kasvulasku} since
$$1<\frac{3}{2}<\frac{5}{2}<\frac{9}{2}<{\dots}.$$}
\end{example}

\begin{example}
{\rm The power series
$$f_2(x)={}_2 F_1(1,1;2;x)=-\frac{\log(1-x)}{x}=
1+\frac{x}{2}+\frac{x^2}{3}+\frac{x^3}{4}+\frac{x^4}{5}+\dots$$ and
$$f_3(x)={}_2 F_1\left(\frac{1}{2},\frac{1}{2};1;x\right)=
\sum_{n\geq0}\frac{\left(\frac{1}{2},n\right)\left(\frac{1}{2},n\right)}{(1,n)n!}x^n=1+\frac{1}{4}x+\frac{9}{64}x^2+\frac{25}{256}x^3+\dots$$
are of type (2) considered in Proposition \ref{kasvulasku} since
$$1>\frac{1}{2}>\frac{1}{3}>\frac{1}{4}>\dots\ \ \  \mbox{and}\ \ \ 1>\frac{1}{4}>\frac{9}{64}>\frac{25}{256}>{\dots}.$$}
\end{example}

\begin{example}
{\rm The power series
$$f_4(x)=1+\frac{77}{80}x+\frac{19}{20}x^2+\frac{3}{2}x^3+\frac{5}{2}x^4+\frac{9}{2}x^5+
\sum_{n\geq 6}\frac{2^{n-2}+1}{2}x^n$$ is of type (2) considered in
Proposition \ref{kasvulasku} since
$$1>\frac{77}{80}>\frac{19}{20}<\frac{3}{2}<\frac{5}{2}<\frac{9}{2}<{\dots}.$$}
\end{example}

Now, let us recall some simple properties of log-convex sequences:
the product and sum of log-convex sequences are also log-convex.
Moreover, it is easy to see that log-convexity is stable under term
by term integration in the following sense: if the coefficients of
the power series $f(x)=\sum_{n\geq0}a_nx^n$ form a log-convex
sequence, then coefficients of the series $$g(x)=\frac{1}{x}
\int_0^x f(t) dt=\sum_{n\geq0}\frac{1}{n+1}a_n x^{n}$$ also form a
log-convex sequence and in view of Theorem \ref{Kaluza1} this
implies that the power series $g(x)$ has also the Kaluza sign
property. On the other hand this is not true about differentiation:
if the coefficients of the series $f(x)=\sum_{n\geq0}a_n x^n$ form a
log-convex sequence, then the coefficients of the power series
$$f'(x)=\sum_{n\geq 0}(n+1) a_{n+1}x^n$$ do not form necessarily a log-convex sequence.
Moreover, it can be shown that if the above power series $f(x)$ has
the Kaluza sign property, then the power series $f'(x)$ does not
need to have the Kaluza sign property.

\begin{example}
{\rm The hypergeometric series
$$f_2(x)=1+\frac{x}{2}+\frac{x^2}{3}+\frac{x^3}{4}+\frac{x^4}{5}+\dots$$ has Kaluza's sign property but the series
$$f_2'(x)=\frac{1}{2}
+ \frac{2}{3}x + \frac{3}{4}x^2 + \frac{4}{5}x^3+\dots$$ does not
since
$$\frac{1}{f_2'(x)}=2 - \frac{8}{3}x + \frac{5}{9}x^2+{\dots}.$$
All the same, the power series
$$\frac{1}{x}\int_0^x f_2(t)dt=1+\frac{x}{4}+\frac{x^2}{9}+\frac{x^3}{16}+\frac{x^4}{25}+{\dots}$$
has the Kaluza sign property.}
\end{example}

The following examples show that if the power series $f(x)$ and
$g(x)$ have Kaluza's sign property, then in general it is not true
that the series $f(x)g(x)$ or the quotient $f(x)/g(x)$ would also
have Kaluza's sign property. Furthermore, if the series $f(x)$ has
the Kaluza sign property, then in general the series
$\left[f(x)\right]^{\alpha}$ does not have the Kaluza sign property
if $\alpha>1.$

\begin{example}
{\rm Let $f_1(x), f_2(x)$ be as earlier. The series $f_1(x)f_2(x)$
and $f_2(x)/f_1(x)$ do not have the Kaluza sign property because
$$\frac{1}{f_1(x)f_2(x)}=1 - 2x +\frac{5}{12}x^2 -\frac{1}{6}x^3-\dots$$
and
$$\frac{1}{f_2(x)/f_1(x)}= 1 + x + \frac{5}{3}x^2 +\frac{37}{12} x^3 +\dots$$}
\end{example}

\begin{example}
{\rm The series $\left[f_1(x)\right]^{3}$ and
$\left[f_2(x)\right]^{1.8}$ do not have the Kaluza sign property
because
$$\frac{1}{\left[f_1(x)\right]^{3}}=1 - \frac{9}{2}x + 6 x^2 - \frac{9}{4}x^3+\dots$$
and
$$\frac{1}{\left[f_2(x)\right]^{1.8}}=1 - 0.9 x + 0.03 x^2 - 0.009 x^3-{\dots}.$$}
\end{example}

\begin{example}
{\rm We note that if the sequence $\{a_n\}_{n\geq0}$ is log-convex
and either $a_0\leq a_1\leq a_2\leq \dots$ or $a_0\geq a_1\geq
a_2\geq \dots,$ then the sequence $\{a_n^{\alpha}\}_{n\geq0}$ would
seem to be also log-convex if $0<\alpha\le1.$ However, if there
exists an index $k\ge1$ such that $a_0\ge a_1\ge a_2\ge {\dots}\ge
a_k\le a_{k+1}\le \dots$ then generally the sequence
$\{a_n^{\alpha}\}_{n\geq0}$ is not log-convex if $0<\alpha <1.$ The
series $f_1(x),f_2(x)$ and $f_3(x)$ are all either of type
$a_0<a_1<a_2<\dots$ or of type $a_0>a_1>a_2>{\dots}.$ Numerical
experiments show that the series
$[f_1(x)]^{\alpha},[f_2(x)]^{\alpha}$ and $[f_3(x)]^{\alpha}$ have
the Kaluza sign property at least for the first 20 terms when
$\alpha=0.05k+0.05$ and $k\in\{0,1,\dots,19\}.$

The series $f_4(x)$ is of type
$a_0>a_1>a_2>\dots>a_k<a_{k+1}<{\dots}.$ The series
$\left[f_4(x)\right]^{1/2}$ does not
have the log-convexity property because
$$\frac{1}{\left[f_4(x)\right]^{1/2}}=
1 + \frac{77}{160} x + \frac{18391}{51200} x^2 + \frac{4727893}{8192000} x^3+
\frac{190367203}{209715200}x^4+\cdots $$ 
and $a_3^2 > a_2 a_4\,.$}
\end{example}

Finally, we note that the coefficients of the Maclaurin series
$$f_5(x)=1+\sum_{n\geq1}\frac{x^n}{n}$$ satisfy (\ref{Kaluza2})
for all $n\in\{2,3,\dots\},$ but the reciprocal power series has a
positive coefficient, that is,
$$\frac{1}{f_5(x)}= 1 - x + \frac{1}{2}x^2 - \frac{1}{3}x^3 + {\dots}. $$
Thus, for the Kaluza sign property it is not enough that
\eqref{Kaluza2} holds starting from some index
$n_0\in\{2,3,\dots\}.$ Moreover, it is not easy to find a series
$f(x)$ whose coefficients would not form a log-convex sequence and
in the series $1/f(x)$ all the coefficients except the constant
would be negative. Hence it seems that log-convexity is near of
being necessary.

Motivated by the above discussion we present the following result.

\begin{theorem}\label{lamp}
Let $f(x)=\sum_{n\geq0}a_nx^n$ and $g(x)=\sum_{n\geq0}b_nx^n$ be two
convergent power series such that $a_n,b_n>0$ for all
$n\in\{0,1,\dots\}$ and the sequences $\{a_n\}_{n\geq0},$
$\{b_n\}_{n\geq0}$ are log-convex. Then the following power series
have the Kaluza sign property:
\begin{enumerate}
\item the scalar multiplication $\alpha f(x)=\sum_{n\geq0}(\alpha
a_n)x^n,$ where $\alpha>0;$
\item the sum $f(x)+g(x)=\sum_{n\geq0}(a_n+b_n)x^n;$
\item the linear combination $\alpha f(x)+\beta g(x)=\sum_{n\geq0}(\alpha a_n+\beta
b_n)x^n,$ where $\alpha,\beta>0;$
\item the Hadamard (or convolution) product $f(x)*g(x)=\sum_{n\geq0}a_nb_nx^n;$
\item $u(x)=\sum_{n\geq0}u_nx^n,$ where $u_n=\sum_{k=0}^nC_n^ka_kb_{n-k};$
\item $v(x)=\sum_{n\geq0}v_nx^n,$ where
$v_n=\sum_{k=0}^n\frac{(\alpha,k)(\beta,n-k)}{k!(n-k)!}a_kb_{n-k}$
and $\alpha,\beta>0$ such that $\alpha+\beta=1.$
\end{enumerate}
\end{theorem}

\begin{proof}[\bf Proof]
Since the sequences $\{a_n\}_{n\geq0}$ and $\{b_n\}_{n\geq0}$ are
positive and log-convex, clearly the sequences $\{\alpha
a_n\}_{n\geq0},$ $\{a_n+b_n\}_{n\geq0},$ $\{\alpha a_n+\beta
b_n\}_{n\geq0}$ and $\{a_nb_n\}_{n\geq0}$ are also positive and
log-convex. Moreover, due to Davenport and P\'olya \cite{dp} we know
that the binomial convolution $\{u_n\}_{n\geq0}$, and the sequence
$\{v_n\}_{n\geq0}$ are also log-convex. Thus, applying Kaluza's
Theorem \ref{Kaluza1}, the proof is complete.
\end{proof}

We note that some related results were proved by Lamperti \cite{l},
who proved among others that if the power series $f(x)$ and $g(x)$
in Theorem \ref{lamp} have the Kaluza sign property, then the power
series $f(x)*g(x)$ and $u(x)$ in Theorem \ref{lamp} have also Kaluza
sign property. With other words the convolution and the binomial
convolution preserve the Kaluza sign property. Lamperti's approach
is different from Kaluza's approach and provides a necessary and
sufficient condition for a power series (with the aid of infinite
matrixes) to have the Kaluza sign property.

\section{Kaluza's criterion and the hypergeometric series}

In this section we give examples of cases of hypergeometric series
when the Kaluza sign property either holds or fails. We shall use
the notation
$${}_2F_1(a,b;c;x)= \sum_{n\geq0}\alpha_nx^n,$$
where $$\alpha_n=\frac{(a,n)(b,n)}{(c,n)n!}.$$

\begin{theorem}\label{hyper1}
If $a,b,c>0,$ $2ab(c+1)\leq(a+1)(b+1)c$ and $c\geq a+b-1,$ then the
sequence $\{\alpha_n\}_{n\geq0}$ is positive and log-convex, and
then the Gaussian hypergeometric series ${}_2F_1(a,b;c;x)$ has the
Kaluza sign property.
\end{theorem}

\begin{proof}[\bf Proof]
To show that the sequence $\{ \alpha_n\}_{n\ge 0}$ is log-convex we
just need to prove that for all $n\in\{1,2,\dots\}$
$$\frac{(a,n)^2(b,n)^2}{(c,n)^2{(n!)}^2}\le\frac{(a,n-1)(b,n-1)}{(c,n-1)(n-1)!}\frac{(a,n+1)(b,n+1)}{(c,n+1)(n+1)!}$$
or equivalently
$$\frac{(a+n-1)(b+n-1)}{(c+n-1)n}<\frac{(a+n)(b+n)}{(c+n)(n+1)}.$$
Now, this is equivalent to the inequality for the second degree
polynomial $$W(n)=w_1 n^2+ w_2n+w_3\ge0,$$ where
$$\left\{\begin{array}{ll}w_1=c+1-a-b\\ w_2=a+b+c-2ab-1\\
w_3=ac+bc-abc-c\end{array}\right.$$ and $n\in\{1,2,\dots\}.$ If
$w_1\ge 0,$ i.e. $c\ge a+b-1,$ then in view of $n^2\geq 2n-1,$ we
obtain that $$W(n)\geq (3c-a-b-2ab+1)n+(ac+bc-abc-2c+a+b-1).$$
Observe that if we suppose $a+b-1-ab>0,$ then $c\geq
a+b-1>(a+b+2ab-1)/3$ and this together with
$2ab(c+1)\leq(a+1)(b+1)c$ imply
\begin{equation}\label{eqwn}W(n)\geq c(a+b-ab+1)-2ab\ge0.\end{equation} On the other hand, if we have
$a+b-1-ab\leq0,$ then because of $2ab(c+1)\leq(a+1)(b+1)c$ we obtain
$a+b+1-ab\geq 2ab/c>0$ and then
$$c\geq\frac{2ab}{a+b+1-ab}\geq ab\geq \frac{a+b+2ab-1}{3},$$ which
implies again \eqref{eqwn}. This completes the proof.
\end{proof}

The next result shows that the condition $2ab(c+1)\leq(a+1)(b+1)c$
in the above theorem is not only sufficient, but even necessary.

\begin{theorem}\label{hyper2}
If $a,b,c>0$ and $2ab(c+1)>(a+1)(b+1)c,$ then the hypergeometric
series ${}_2F_1(a,b;c;x)$ does not have the Kaluza sign property.
\end{theorem}

\begin{proof}[\bf Proof]
Suppose that the coefficient $a_n$ are defined by
$$ \frac{1}{\sum_{n\ge 0}\frac{(a,n)(b,n)}{(c,n)n!}x^n}=1+a_1 x+a_2 x^2+a_3 x^3+{\dots}.$$
Then $$a_1=-\frac{ab}{c},\
a_2=-\frac{ab}{c}a_1-\frac{a(a+1)b(b+1)}{c(c+1)2}=\frac{ab}{c}
\left(\frac{ab}{c}-\frac{(a+1)(b+1)}{(c+1)2}\right).$$ We shall only
look at the sign of $a_2.$ If $a_2>0$ then ${}_2F_1(a,b;c;x)$ does
not have Kaluza's sign property. With this the proof is complete.
\end{proof}

For Theorem \ref{hyper2} we now give an illuminating example.

\begin{example}
{\rm If we consider the hypergeometric series $${}_2F_1(3,3;6;x)=1 +
\frac{3}{2}x + \frac{12}{7}x^2 + \frac{25}{14}x^3+
\frac{25}{14}x^4+\dots$$ and look at its reciprocal series we get a
positive coefficient different from a constant term
$$\frac{1}{{}_2F_1(3,3;6;x)}=1 - \frac{3}{2}x + \frac{15}{28}x^2 - \frac{1}{56}x^3 +{\dots}.$$}
\end{example}

Next we are going to present a counterpart of Theorem \ref{hyper1}.
To do this we first recall the following result of Jurkat \cite{j}.

\begin{theorem}\label{thjurk}
Let us consider the power series $p(x)=\sum_{n\geq0}p_nx^n$ and
$q(x)=\sum_{n\geq0}q_nx^n,$ where $p_0>0$ and the sequence
$\{p_n\}_{n\geq0}$ is decreasing. If for all $n\in\{1,2,\dots\}$
\begin{equation}\label{jurkateq}\overline{\Delta}q_n\geq
\frac{q_0}{p_0}\overline{\Delta}p_n,\end{equation} where
$\overline{\Delta}a_n=a_n-a_{n-1}$ for all $n\in\{1,2,\dots\},$
$\overline{\Delta}a_0=a_0,$ then the coefficients of the power
series $k(x)=q(x)/p(x)=\sum_{n\geq0}k_nx^n$ satisfies $k_n\geq0$ for
all $n\in\{1,2,\dots\}.$ Moreover, if \eqref{jurkateq} is reversed,
then $k_n\leq0$ for all $n\in\{1,2,\dots\}.$
\end{theorem}

Note that the first part of the above result is \cite[Theorem 4]{j},
while the second is \cite[Theorem 5]{j}. First, let us consider in
the above theorem $q_0=1$ and $q_n=0$ for all $n\in\{1,2,\dots\}$ to
have $k(x)=1/p(x),$ as in \cite[Theorem 3]{j}. Then the condition
$q_n-q_{n-1}\geq (q_0/p_0)(p_n-p_{n-1}),$ i.e. \eqref{jurkateq} for
$n=1$ means that $p_1\leq0$ and for $n\in\{2,3,\dots\}$ means that
$p_n\leq p_{n-1}.$ Thus, we obtain the following result.

\begin{proposition}\label{propo}
If $a_0>0\geq a_1\geq a_2\geq{\dots}\geq a_n\geq{\dots},$ then the
reciprocal of the power series $f(x)=\sum_{n\geq 0}a_nx^n$ has all
coefficients non-negative. More precisely, if
$1/f(x)=\sum_{n\geq0}b_nx^n,$ then $b_n\geq0$ for all
$n\in\{0,1,\dots\}.$
\end{proposition}

By using the above result we may get the following.

\begin{theorem}
If $a,b,c>-1,$ $c\neq0,$ $ab/c\leq0,$ and $c\leq\min\{a+b-1,ab\},$
then the reciprocal of the series ${}_2F_1(a,b;c;x)$ has all
coefficients non-negative, that is, we have
$1/{}_2F_1(a,b;c;x)=1+\sum_{n\geq1}\beta_nx^n$ with $\beta_n\geq0$
for all $n\in\{1,2,\dots\}.$
\end{theorem}

\begin{proof}[\bf Proof]
Clearly $\alpha_0=1>0$ and $\alpha_1=ab/c\leq0.$ The condition
$\alpha_n\geq\alpha_{n+1}$ holds for all $n\in\{1,2,\dots\}$ if and
only if we have
$$\frac{(a,n)(b,n)}{(c,n+1)(n+1)!}\left((c+n)(n+1)-(a+n)(b+n)\right)\geq0$$
for all $n\in\{0,1,\dots\}.$ Now, because $a,b,c>-1,$ $c\neq0$ and
$ab/c\leq0,$ for all $n\in\{0,1,\dots\}$ we should have
$$(a+b-c-1)n+ab-c\geq0$$
Applying Proposition \ref{propo}, the result follows.
\end{proof}

Now, let us focus on the second part of Theorem \ref{thjurk}, i.e.
\cite[Theorem 5]{j}. Consider again $q_0=1$ and $q_n=0$ for all
$n\in\{1,2,\dots\}$ to have $k(x)=1/p(x),$ as above. Then the
condition $q_n-q_{n-1}\leq (q_0/p_0)(p_n-p_{n-1})$ for $n=1$ means
that $p_1\geq0$ and for $n\in\{2,3,\dots\}$ means that $p_n\geq
p_{n-1},$ which contradicts condition \cite[Eq. (6)]{j}, i.e. the
hypothesis that the sequence $\{p_n\}_{n\geq0}$ is decreasing.
However, following the proof of \cite[Theorem 4]{j}, it is easy to
see that to have a correct version of \cite[Theorem 5]{j} we need to
assume that the sequence $\{q_n\}_{n\geq0}$ is strictly decreasing.
More precisely, with the notation of Theorem \ref{thjurk} we have
$$q_n=\sum_{i=0}^{n}k_ip_{n-i},$$
and then
$$q_n-q_{n-1}=k_0(p_n-p_{n-1})+\sum_{i=1}^{n-1}k_i(p_{n-i}-p_{n-i-1})+k_np_0.$$
which can be rewritten in the form
$$k_np_0=\overline{\Delta}q_n-\frac{q_0}{p_0}\overline{\Delta}p_n-\sum_{i=1}^{n-1}k_i(p_{n-i}-p_{n-i-1}).$$
Now, suppose that $k_1,k_2,\dots,k_{n-1}\leq0.$ Since
$\{p_n\}_{n\geq0}$ is decreasing, we obtain
$$k_np_0\leq\overline{\Delta}q_n-\frac{q_0}{p_0}\overline{\Delta}p_n$$
which is clearly non-positive if the reversed form of
\eqref{jurkateq} holds. However, here it is very important to note
that if $\overline{\Delta}q_n\geq0,$ then the right-hand side of the
above expression is non-negative. Summarizing, in the second part of
Theorem \ref{thjurk} we need to suppose that the sequence
$\{q_n\}_{n\geq0}$ is strictly decreasing.

\section{The monotonicity of the quotient of two hypergeometric series}

The next result, due to M. Biernacki and and J. Krzy\.z, has found
numerous applications during the past decade. For instance in
\cite{hvv} the authors give a variant of Theorem \ref{hyper3} where
the numerator and denominator Maclaurin series are replaced with
polynomials of the same degree. See also \cite{baricz} for an
alternative proof of Theorem \ref{hyper3} and \cite{bar} for some
interesting applications.

\begin{theorem}\label{hyper3}
Suppose that the power series $f(x)=\sum_{n\geq0}a_nx^n$ and
$g(x)=\sum_{n\geq0}b_n x^n$ have the radius of convergence $r>0$ and
$b_n>0$ for all $n\in\{0,1,\dots\}.$ Then the function $x\mapsto
{f(x)}/{g(x)}$ is increasing (decreasing) on $(0,r)$ if the sequence
$\{a_n/b_n\}_{n\geq0}$ is increasing (decreasing).
\end{theorem}

Now, with the help of Theorem \ref{hyper3} we prove the following,
which completes \cite[Theorem 3.8]{hvv}.

\begin{theorem}\label{hyper4}
Let $a_1,a_2,b_1,b_2,c_1,c_2$ be positive numbers. Then the series
$$x\mapsto q(x)=\frac{{}_2F_1(a_1,b_1;c_1;x)}{{}_2F_1(a_2,b_2;c_2;x)}=
\frac{r_0 + r_1 x + r_2 x^2+\dots}{s_0 + s_1 x + s_2 x^2+\dots}$$
is increasing on $(0,1)$ if one of the following conditions holds
\begin{enumerate}
\item[(1)] $a_1 \ge a_2,$ $b_1 \ge b_2$ and $c_2 \ge c_1.$
\item[(2)] $a_1+b_1 \ge a_2+b_2,$ $c_2\ge c_1$ and $a_2\le a_1\le b_1\le
b_2.$
\item[(3)] $a_1+b_1 \ge a_2+b_2,$ $c_2\ge c_1$ and $a_1b_1\ge a_2b_2.$
\end{enumerate}
Moreover, if the above inequalities are reversed, then the function
$x\mapsto q(x)$ is decreasing on $(0,1).$
\end{theorem}

\begin{proof}[\bf Proof]
We prove only the part when $x\mapsto q(x)$ is increasing. The other
case is similar, so we omit the details. Observe that the sequence
$\{r_n/s_n\}_{n\geq0}$ is increasing if and only if for all
$n\in\{0,1,\dots\}$ we have
$$\frac{r_n}{s_n} = \frac{\frac{(a_1,n)(b_1,n)}{(c_1,n)n!}}{\frac{(a_2,n)(b_2,n)}{(c_2,n)n!}} \le
\frac{\frac{(a_1,n+1)(b_1,n+1)}{(c_1,n+1)(n+1)!}}{\frac{(a_2,n+1)(b_2,n+1)}{(c_2,n+1)(n+1)!}}
= \frac{r_{n+1}}{s_{n+1}}$$ or equivalently
\begin{equation}\label{quo}(a_2+n)(b_2+n)(c_1+n)\le
(a_1+n)(b_1+n)(c_2+n).\end{equation}

(1) By using the previous theorem we get both cases of the first
claim.

(2) For the second claim we only need to prove that
$(a_2+n)(b_2+n)\le (a_1+n)(b_1+n)$ for all $n\in\{0,1,\dots\}.$ We
can reduce $a_1$ and $b_1$ into $a_1'$ and $b_1'$ so that
$a_1'+b_1'=a_2+b_2$ and $0< a_2\le a_1'\le b_1'\le b_2$ still holds.
Now we get both cases of the second claim by noticing that the graph
of the function $f(t)=(a_2+b_2+n-t)(n+t)$ is a parabola which gets
its maximum value in $(a_2+b_2)/2$ and that $f(a_2)\le f(a_1').$

(3) Observe that if $a_1b_1\geq a_2b_2$ and $a_1+b_1\geq a_2+b_2,$
then $$n^2+(a_1+b_1)n+a_1b_1\geq n^2+(a_2+b_2)n+a_2b_2$$ or
equivalently $$(a_2+n)(b_2+n)\le (a_1+n)(b_1+n)$$ for all
$n\in\{0,1,\dots\}.$
\end{proof}

Now, we would like to study the sign of the coefficients of the
power series $q(x)$ in Theorem \ref{hyper4}. However, it is not easy
to use Jurkat's result in Theorem \ref{thjurk}, since it is
difficult to verify for what $a_1,b_1,c_1,a_2,b_2$ and $c_2$ is
valid the inequality $r_n-r_{n-1}\geq s_n-s_{n-1}$ or its reverse
for all $n\in\{1,2,\dots\}.$ All the same, there is another useful
result of Jurkat \cite{j}, which generalizes Kaluza's Theorem
\ref{Kaluza1} and it is strongly related to Theorem \ref{hyper3} of
Biernacki and Krzy\.z.

\begin{theorem}\label{thjurkat}
Let us consider the power series $f(x)=\sum_{n\geq0}a_nx^n$ and
$g(x)=\sum_{n\geq0}b_nx^n,$ where $b_n>0$ for all
$n\in\{0,1,\dots\}$ and the sequence $\{b_n\}_{n\geq0}$ is
log-convex. If the sequence $\{a_n/b_n\}_{n\geq0}$ is increasing
(decreasing), then the coefficients of the power series
$q(x)=f(x)/g(x)=\sum_{n\geq0}q_nx^n$ satisfies $q_n\geq0$
($q_n\leq0$) for all $n\in\{1,2,\dots\}.$
\end{theorem}

It is important to note here that if the radius of convergence of
the above power series is $r,$ as above, then clearly the conditions
of the above theorem imply the monotonicity of the quotient $q.$
Thus, combining Theorem \ref{hyper1} with Theorem \ref{thjurkat} we
obtain the following result.

\begin{theorem}
Suppose that all the hypotheses of Theorem \ref{hyper4} are
satisfied and in addition $2a_2b_2(c_2+1)\leq(a_2+1)(b_2+1)c_2$ and
$c_2\ge a_2+b_2-1.$ Then the coefficients of the quotient
$$x\mapsto q(x)=\frac{{}_2F_1(a_1,b_1;c_1;x)}{{}_2F_1(a_2,b_2;c_2;x)}=
\frac{r_0 + r_1 x + r_2 x^2+\dots}{s_0 + s_1 x + s_2
x^2+\dots}=q_0+q_1x+q_2x^2+\dots$$ satisfy $q_n\geq0$ for all
$n\in\{1,2,\dots\}.$ Moreover, if the inequalities in Theorem
\ref{hyper4} are reversed, then $q_n\leq0$ for all
$n\in\{1,2,\dots\}.$

 \end{theorem}

Rational expressions involving hypergeometric functions occur in
many contexts in classical analysis. For instance \cite[Theorem
3.21]{avv} states some properties such as monotonicity or convexity
of several functions of this type. But much stronger conclusions
might be true. In fact, in \cite[p. 466]{avv} it is suggested that
several of the functions in the long list of \cite[Theorem
3.21]{avv} might have Maclaurin series with coefficients of the same
sign (except possibly the leading coefficient). This topic remains
widely open since there does not seem to exist a method for
approaching this type of questions.

Finally, let us mention another result, which is also strongly
related to Biernacki and Krzy\.z criterion and is useful in
actuarial sciences in the study of the non-monotonic ageing property
of residual lifetime.

\begin{theorem}
Suppose that the power series $f(x)=\sum_{n\geq0}a_nx^n$ and
$g(x)=\sum_{n\geq0}b_n x^n$ have the radius of convergence $r>0.$ If
the sequence $\{a_n/b_n\}_{n\geq0}$ satisfies $a_0/b_0\leq
a_1/b_1\leq \dots\leq a_{n_0}b_{n_0}$ and $a_{n_0}b_{n_0}\geq
a_{n_0+1}b_{n_0+1}\geq\dots\geq a_nb_n\geq\dots$ for some
$n_0\in\{0,1,\dots,n\},$ then there exists an $x_0\in(0,r)$ such
that the function $x\mapsto {f(x)}/{g(x)}$ is increasing on
$(0,x_0)$ and decreasing on $(x_0,r).$
\end{theorem}

Note the a variant of the above result appears recently in
\cite[Lemma 6.4]{ruiz} with $a_n$ and $b_n$ replaced with $a_n/n!$
and $b_n/n!$ and the proof is based on the so-called variation
diminishing property of totally positive functions in the sense of
Karlin.

\subsection*{Acknowledgments}
The research of \'Arp\'ad Baricz was supported by the J\'anos Bolyai
Research Scholarship of the Hungarian Academy of Sciences and by the
Romanian National Council for Scientific Research in Education
CNCSIS-UEFISCSU, project number PN-II-RU-PD\underline{ }388/2011.
The research of Matti Vuorinen was supported by the Academy of Finland,
Project 2600066611.
The authors are indebted to the referee for his/her constructive
comments and helpful suggestions, which improved the first draft of
this paper.

\end{document}